\pdfoutput=1
\documentclass[a4paper,reqno, 12pt]{amsart}
\input{psfig.sty}

\usepackage[utf8]{inputenc}
\usepackage{microtype}

\usepackage{amssymb, graphicx, enumerate, enumitem, color}
\usepackage[usenames,dvipsnames,svgnames,table]{xcolor}
\usepackage[foot]{amsaddr}
\usepackage{xspace}
\usepackage{tikz}

\usepackage[ruled,titlenumbered,linesnumbered,noend,norelsize]{algorithm2e}
\DontPrintSemicolon
\SetNlSty{footnotesize}{}{}
\IncMargin{0.5em}
\SetNlSkip{1em}
\SetKwIF{If}{ElseIf}{Else}{if}{then}{else if}{else}{endif}
\SetKw{Return}{return}
\SetNlSkip{2em}
\SetKwProg{Proc}{Procedure}{}{}

\newenvironment{algorithm-hbox}{\hbadness=10000\begin{algorithm}}{\end{algorithm}}

\advance\parskip3pt

\colorlet{mylinkcolor}{violet}
\colorlet{mycitecolor}{YellowOrange}
\colorlet{myurlcolor}{Aquamarine}
\usepackage[unicode=true]{hyperref}
\hypersetup{
    pdftitle={Boolean dimension and tree-width},
    pdfauthor={Stefan Felsner, Tam\'{a}s M\'{e}sz\'{a}ros, and Piotr Micek},
    colorlinks,
    linkcolor={red!50!black},
    citecolor={blue!50!black},
    urlcolor={blue!80!black}
}

\newtheorem{theorem}{Theorem}

\newtheorem{problem}[theorem]{Problem}

\newtheorem*{claim}{Claim}
\newtheorem{lemma}[theorem]{Lemma}
\newtheorem{tool}{Tool}
\theoremstyle{remark}

\newcommand{\set}[1]{\{#1\}}

\newcommand{\calB}{\mathcal{B}}

\newcommand{\calF}{\mathcal{F}}
\newcommand{\calG}{\mathcal{G}}

\newcommand{\red}{\color{red}}

\DeclareMathOperator\tw{tw}
\DeclareMathOperator\td{td}

\let\leq\leqslant
\let\geq\geqslant

\let\subset\subseteq

\let\epsilon\varepsilon

\let\setminus\smallsetminus

\renewenvironment{enumerate}{\begin{enumorig}[label=\textup{(\roman*)}, noitemsep, topsep=2pt plus 2pt, labelindent=.2em, leftmargin=*, widest=iii]}{\end{enumorig}}

\renewenvironment{itemize}{\begin{itemorig}[label=\textbullet, noitemsep, topsep=2pt plus 2pt, labelindent=.5em, labelsep=.5em, leftmargin=*]}{\end{itemorig}}

\makeatletter
\let\old@setaddresses\@setaddresses
\def\@setaddresses{\bigskip\bgroup\parindent 0pt\let\scshape\relax\old@setaddresses\egroup}
\makeatother

\renewcommand{\root}[1]{\operatorname{root}(#1)}

\newcommand{\ve}{\operatorname{vec}}
\newcommand{\rep}[2]{\operatorname{repr}_{#1}(#2)}
\newcommand{\act}[1]{\operatorname{Active}(#1)}

\newcommand{\tree}{\operatorname{tree}}
\newcommand{\bdim}[1]{\operatorname{bdim}(#1)}
\newcommand{\eval}[1]{\operatorname{eval}(#1)}

\newcommand{\proj}[1]{\operatorname{proj}(#1)}

\renewcommand{\red}{\textsc{Red}\xspace}
\newcommand{\green}{\textsc{Green}\xspace}

\begin{document}
\title[Boolean dimension and tree-width]{Boolean dimension and tree-width}

\author[S.~Felsner]{Stefan Felsner}
\address[S.~Felsner]{Institut f\"ur Mathematik, 
Technische Universit\"{a}t Berlin
	}
\email{felsner@math.tu-berlin.de}

\author[T.~M\'{e}sz\'{a}ros]{Tam\'{a}s M\'{e}sz\'{a}ros}
\address[T.~M\'{e}sz\'{a}ros]{Institut f\"ur Mathematik,	
Freie Universit\"at Berlin
	}
\email{tamas.meszaros@fu-berlin.de}

\author[P.~Micek]{Piotr Micek}
\address[P.~Micek]{
Faculty of Mathematics and Computer Science, 
Jagiellonian University 
	}
\email{piotr.micek@tcs.uj.edu.pl}

\thanks{Piotr Micek was partially supported by the National Science Center of Poland, grant no.\ 2015/18/E/ST6/00299.}
\thanks{Tamás Mészáros was supported by the Dahlem Research School of Freie Universit\"at Berlin.}



\keywords{posets, tree-width, boolean dimension}

\begin{abstract}
Dimension is a key measure of complexity of partially ordered sets.
Small dimension allows succinct encoding. Indeed if $P$ has dimension
$d$, then to know whether $x \leq y$ in $P$ it is enough to check whether
$x\leq y$ in each of the $d$ linear extensions of a witnessing realizer.
Focusing on the encoding aspect, Ne\v{s}et\v{r}il and Pudl\'{a}k defined a more expressive version of dimension. 
A poset $P$ has boolean dimension at most $d$ if it is possible to
decide whether $x \leq y$ in $P$ by looking at the relative position of $x$ and
$y$ in only $d$ linear orders on the elements of $P$ (not necessarilly linear extensions).
We prove that posets with cover graphs
of bounded tree-width have bounded boolean dimension.
This stands in contrast with the fact that there are posets with cover graphs of tree-width three and arbitrarily large dimension.
This result might be a step towards a resolution of the long-standing open problem:
Do planar posets have bounded boolean dimension?
\end{abstract}
\maketitle

\section{Introduction}

Partially ordered sets, called \emph{posets} for short, are
combinatorial structures with applications in various mathematical
fields, e.g.\ set theory, topology, algebra, and theoretical
computer science.  The most important measure of a poset's complexity is
its dimension.  

A  \emph{linear extension} $L$ of a poset $P$ is a total order on the elements of $P$ such that if $x\leq y$ in $P$ 
then $x\leq y$ in $L$. A \emph{realizer} of a poset $P$ is a set $\set{L_1,\ldots,L_d}$ of
linear extensions of $P$ such that
\[
x \leq y\Longleftrightarrow (x\leq y\text{ in $L_1$}) 
\wedge \cdots \wedge (x\leq y\text{ in $L_d$}),
\]
for every $x,y\in P$.
The \emph{dimension} of $P$, denoted by $\dim(P)$, is the minimum size of its realizer.

The realizers provide a way of succinctly encoding posets. Indeed, if a poset 
is given with a realizer witnessing dimension $d$, then a 
query of the form "is $x\leq y$?" can be answered by looking at the
relative position of $x$ and $y$ in each of the $d$ linear extensions
of the realizer. This application motivates the
following presumably more efficient encoding of posets proposed by Ne\v{s}et\v{r}il and Pudl\'ak \cite{NP89}, following the work of
Gambosi, Ne\v{s}et\v{r}il and Talamo~\cite{GNT90}.

A \emph{boolean realizer} of a poset $P$ is a set of linear orders (not necessarily linear extensions) $\set{L_1,\ldots,L_d}$ of elements of $P$
for which there exists a $d$-ary boolean formula $\phi$ such that
\[
x \leq y \Longleftrightarrow \phi((x\leq y\text{ in $L_1$}),\ldots,(x\leq y\text{ in $L_d$}))=1,
\]
for every $x,y \in P$. 
The \emph{boolean dimension} of $P$, denoted by $\bdim{P}$, is a
minimum size of a boolean realizer.
Clearly, for every
poset $P$ we have
\[
\bdim{P} \leq \dim(P).
\]

The usual dimension of a poset on $n$ elements may be linear in $n$.
The so-called, \emph{standard example} $S_n$, for $n\geq2$, is a poset
on $2n$ elements $a_1,\ldots,a_n, b_1,\ldots,b_n$ with $a_i < b_j$
in $S_n$ if and only if $i\neq j$, and with no other comparabilities
(see Figure~\ref{fig:stand5}).  On the other hand, in their seminal
paper Dushnik and Miller~\cite{DM41} observed that $\dim(S_n) = n$.  
It is a nice little exercise to show that $\bdim{S_n}\leq4$ for every $n$.
In general, Ne\v{s}et\v{r}il and Pudl\'{a}k~\cite{NP89}
showed that boolean dimension of posets on $n$ elements is
$O(\log{n})$. They also provide an easy counting argument showing that
there are posets on $n$ elements with boolean dimension at least
$c\log n$, where $c$ is some positive constant.

\begin{figure}
\centering
\includegraphics[scale=.3]{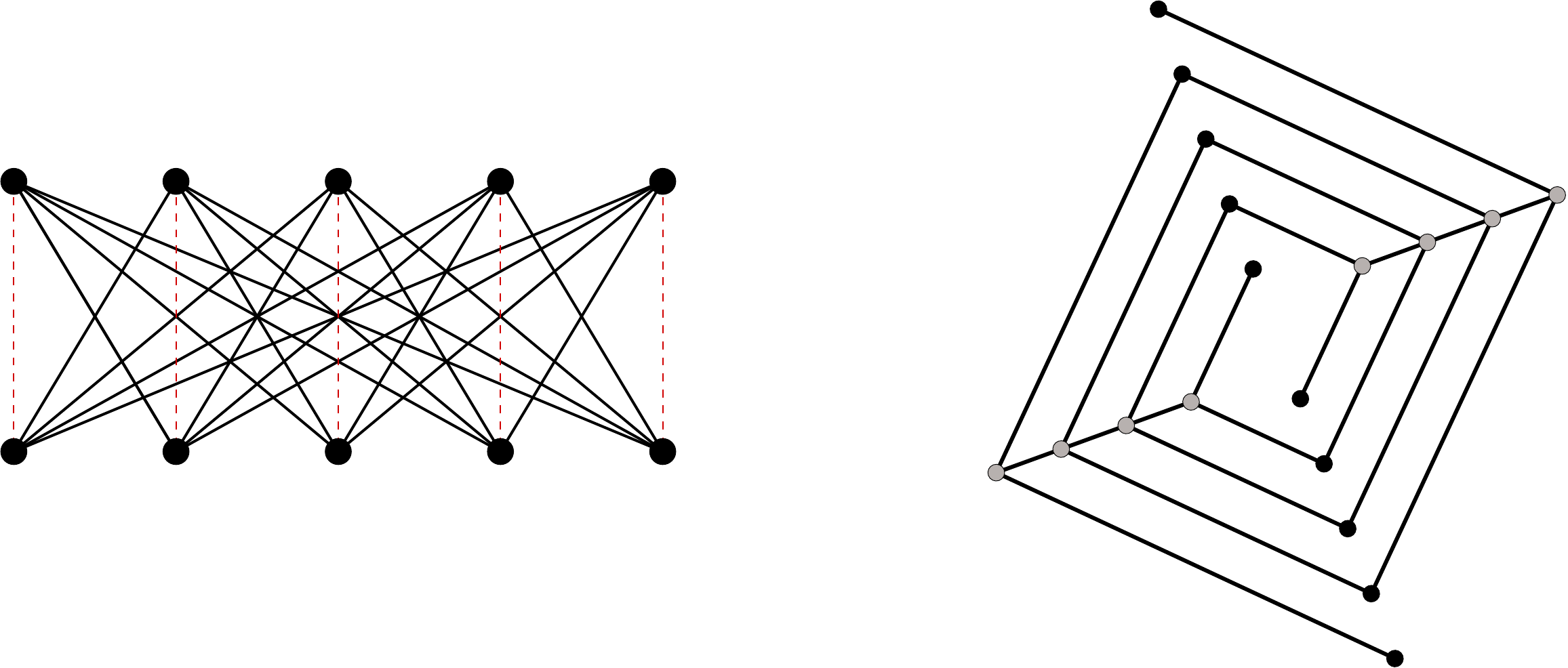}
\caption{\label{fig:stand5}The standard example $S_5$ (left). Kelly's
  planar poset containing an induced $S_5$ (right).}
\end{figure}

The \emph{cover graph} of a poset $P$ is the graph on the
elements of $P$ with edge set $\set{xy\mid x<y \text{ in $P$ and there
    is no $z$ with $x<z<y$ in $P$}}$.  A poset is \emph{planar} if it
has a planar diagram, i.e., its cover graph has a non-crossing
upward drawing in the plane. This means that every edge $xy$ with
$x<y$ is drawn as a curve that goes monotonically up from the point of
$x$ to the point of $y$.  Somewhat unexpectedly, planar posets have
arbitrarily large dimension.  Kelly~\cite{Kel81} gave a construction
that embeds a standard example as a subposet into a planar poset, see
Figure~\ref{fig:stand5}.  Another property of Kelly's construction is
that the cover graphs of resulting posets have tree-width (and even
path-width) at most $3$.  Still, the boolean dimension of standard
examples and Kelly's construction is at most $4$.  There is a
beautiful open problem posed in~\cite{NP89} that remains a challenge
with essentially no progress over the years.
\begin{problem}[Ne\v{s}et\v{r}il and Pudl\'{a}k (1989)]
	\label{problem:planar}
	Is the boolean dimension of planar posets bounded?
\end{problem}

Ne\v{s}et\v{r}il and Pudl\'{a}k suggested an approach for the negative resolution of this problem that involves an auxiliary Ramsey-type problem for planar posets.
From the positive side, Brightwell and Franciosa~\cite{BF96} in 1996 proved that spherical posets with
a least element have bounded boolean dimension, contrary to ordinary dimension.  
More recently, there was a considerable effort to understand the general problem 
whether posets with stucturally restricted cover graphs have bounded boolean dimension.
M\'{e}sz\'{a}ros, Micek and Trotter~\cite{MMT} proved that boolean dimension of a poset 
is bounded in terms of boolean dimension of its $2$-connected blocks.
Micek and Walczak~\cite{MW-pw} proved that 
posets with cover graphs of bounded path-width have bounded boolean dimension.

The contribution of our paper is the following result.

\begin{theorem}\label{thm:main}
  Posets with cover graphs of bounded tree-width have bounded boolean
  dimension.
\end{theorem}

The usual dimension is known to be at most $3$ for posets with cover graphs being forests (Trotter, Moore~\cite{TM77}) and at most $12$ for posets with cover graphs of tree-width $2$ (Seweryn~\cite{Sew}).
As mentioned before, Kelly's examples have tree-width $3$ and arbitrarily large dimension.
This certifies that boolean realizers are able to represent natural classes of posets that are out of reach in the default setting.

It is tempting to ask whether a result similar to Theorem~\ref{thm:main} holds for broader classes of sparse posets.
Besides planar posets, it might be true even for posets whose cover graphs exclude a fixed graph as a minor. However, we note that excluding a fixed graph as a topological minor is insufficient. Indeed, there are posets whose cover graphs have maximum degree at most $3$ and that have unbounded boolean dimension. For completeness we include such an example in Section \ref{sec:topminor}.

In 2016, Ueckerdt~\cite{Uec16} proposed yet another variant of poset's dimension, the \emph{local dimension}.
Interestingly, local dimension of posets with bounded path-width is bounded, while 
it becomes unbounded for posets of tree-width $3$ (Barrera-Cruz et al.~\cite{BPSTT}).
It is also known that planar posets have unbounded local dimension 
(Bosek, Grytczuk and Trotter~\cite{BGT}).

This paper is organized as follows. In Section~\ref{sec:preliminaries}, we proceed with the necessary definitions.
In particular, we introduce branching programs which we will use to build our formulas in the boolean realizers. 
We present two simple but important subroutines that we use later on extensively.
In Section~\ref{sec:families}, we set up the proof of Theorem~\ref{thm:main} and build some auxiliary structures and colorings based on the poset given on the input and on the tree-decomposition of its cover graph.
In Section~\ref{sec:main-program}, we prove Theorem~\ref{thm:main}, while
in Section~\ref{sec:reachability}, we present a connection of boolean dimension to labeling schemes for reachability queries.
In particular, we discuss how a positive resolution of Problem~\ref{problem:planar} would imply the existence of a labeling scheme of size $O(\log n)$ for reachability queries for planar digraphs. Finally, in Section~\ref{sec:topminor}, we discuss an example of posets whose cover graphs have maximum degree at most $3$ and that have unbounded boolean dimension

\section{Preliminaries}\label{sec:preliminaries}

\subsection{Tree-decompositions}

Let $G$ be a graph. 
A \emph{tree-decomposition} of $G$ is a pair $(T, \{B_t\}_{t\in V(T)})$ where $T$ is a tree and $\{B_t\}_{t\in V(T)}$ is a family of subsets of $V(G)$ satisfying:
\begin{enumerate}
\item for each $v\in V(G)$ there exists $t\in V(T)$ with $v\in B_t$;
\item for every edge $uv\in E(G)$ there exists $t\in V(T)$ with $u,v\in B_t$;
\item\label{twprop3} for each $v\in V(G)$, if $v\in B_t \cap B_{t''}$ for some $t,t''\in V(T)$, and $t'$ lies on the path in $T$ between $t$ and $t''$, then $v\in B_{t'}$.
\end{enumerate}

By Property \ref{twprop3}, we have that for every vertex $v\in V(G)$, the vertices $t\in V(T)$ for which $v\in B_t$ form a subtree of $T$, called the \emph{subtree of $v$}.

The quality of a tree-decomposition $(T, \{B_t\}_{t\in V(T)})$ is usually measured by its
width, i.e. the maximum of $|B_t|-1$ over all $t\in V(T)$. 
Then the \emph{tree-width} $\text{tw}(G)$ of G is the
minimum width of a tree-decomposition of G.

\subsection{Branching programs}

To show that a poset $P$ has $\bdim P\leq d$, one should provide~$d$
linear orders $\pi_1,\dots,\pi_d$ of the elements of $P$ and a boolean
formula $\phi(\xi_1,\dots,\xi_d)$. It will be convenient to describe
the formula $\phi$ as a branching program.

We think of a branching program $\calB$ as a rooted tree.
The nodes of the tree are subprograms.
The task of a subprogram at node $N$ is to return a boolean value 
to its parent, this boolean value is called the \emph{evaluation} of $N$.
The evaluation of $N$ is a function of the 
values returned from the children and of the input variables $\xi_1,\dots,\xi_d$ of $\calB$.
The evaluation $\eval {\calB}$ of the branching program $\calB
(b_1,\dots,b_d)$ with inputs $b_1,\ldots,b_d$ is the evaluation of its root node. A branching program $\calB$ is said to represent a boolean formula
$\phi(\xi_1,\dots,\xi_d)$ if for all inputs $b_1,\dots,b_d$, we have
$\eval {\calB}=\phi(b_1,\dots,b_d)$. 

Therefore we can prove $\bdim P\leq d$ by describing a
branching program that answers queries `$(x<y)?$` using variables
$\xi_1,\dots,\xi_d$, where $\xi_i=1$ if and only if $x<y$ in $\pi_i$.
In this case, we say that a branching program \emph{depends} on $\pi_1,\dots,\pi_d$.

\subsection{Tools}

The first tool is a straightforward branching program able to detect if a queried element lies within some fixed set.
For a linear order $\pi$ on a set $V$ let $\pi^*$ denote the reversal of $\pi$ and for $X\subset V$, let $\pi(X)$ denote the projection of $\pi$ onto $X$.
We also work with linear orders as sequences, i.e., when $A$ is disjoint from $B$, $\pi$ is a linear order on $A$ and $\sigma$ is a linear order on $B$, then $\pi\sigma$ is a concatenation of these linear orders which is a linear order on $A \cup B$.

\begin{tool}[Set Membership]\label{cor:set-memb}
Let $V$ be a set and let $C\subseteq V$. 
Then there are three linear orders on $V$ such
that for every distinct $x,y\in V$ by looking at the order of $x$ and~$y$ in
these three linear orders one can decide whether each $x$ and $y$ belong
to $C$ or not.
\end{tool}
\begin{proof}
Fix an arbitrary linear order $\pi$ on $V$ and consider the following three linear orders: $\pi_1=\pi(C)\pi(V-C)$, $\pi_2=\pi^*(C)\pi(V-C)$ and $\pi_3=\pi(C)\pi^*(V-C)$ do the job. Then $x\in C$ if and only if $x$ and $y$ switch orders in $\pi_1$ and $\pi_2$ or in all three linear orders $x<y$. Similarly, $y\in C$ if and only if $x$ and $y$ switch orders in $\pi_1$ and $\pi_2$ or in all three linear orders $y<x$.
\end{proof}

The other tool deals with elements lying in a tree.
Let $T$  be a rooted tree with some of its edges being colored either with \red or with \green. 
For a vertex $u\in V(T)$, we denote the subtree of $T$ rooted at $u$ by $T_u$. If $v\in T_u$, then we say that $u$ is \emph{below} $v$ in $T$ or $v$ is \emph{above} $u$ in $T$. Note that, in particular, $u$ is below itself.
For $u,v\in V(T)$ we denote the unique path between $u$ and $v$ in $T$ by $[u,v]$, 
and the vertex in $[u,v]$ that is closest to the root of $T$ by $u\wedge v$.  
We call $u\wedge v$ the \emph{meet} of $u$ and $v$. 
We also assume that $T$ is given together with a planar upward drawing with lowest vertex being the root. 
This implies that at every vertex $u$ there is a left-to-right ordering of the subtrees rooted at the children of $u$. 
Now suppose that the vertices $u$ and $v$ are such that none of them is below the other. 
Then $u\wedge v$ has two children $u'\neq v'$ such that $u\in T_{u'}$ and $v\in T_{v'}$. 
We say that $u$ \emph{is  left of} $v$ in $T$ if $T_{u'}$ comes before $T_{v'}$ in the left-to-right ordering of the subtrees rooted at the children of $u\wedge v$. For example, in Figure~\ref{fig:red-green tree} we have that $c$ is below $r$, the meet of $f$ and $s$ is $a$, and $k$ is left of $n$.

\begin{figure}[!h]
	\begin{center}
		\begin{tikzpicture}[scale=.73,every node/.style={circle,draw,color=black,fill=black,inner sep=0pt,minimum width=4pt}]
			\begin{scope}[shift={(0,-2)}]
	      		\node (a) at (0,0) [label=below:$a$]{};
	
	      		\node (b) at (-7.5,2) [label=below:$b$]{};
	      		\node (c) at (-2.5,2) [label=below:$c$]{};
	      		\node (d) at (2.5,2) [label=below:$d$]{};
	      		\node (e) at (7.5,2) [label=below:$e$]{};

	      		\node (f) at (-9.5,4) [label=left:$f$]{};
	      		\node (g) at (-7.5,4) [label=left:$g$]{};
	      		\node (h) at (-5.5,4) [label=right:$h$]{};
	      		
	      		\node (i) at (-3.5,4) [label=left:$i$]{};
	      		\node (j) at (-1.5,4) [label=right:$j$]{};
	      		
	      		\node (k) at (0.5,4) [label=left:$k$]{};
	      		\node (l) at (2.5,4) [label=left:$l$]{};
	      		\node (m) at (4.5,4) [label=right:$m$]{};
	      		
	      		\node (n) at (6.5,4) [label=left:$n$]{};
	      		\node (o) at (8.5,4) [label=right:$o$]{};
	      		
	      		\node (p) at (-7.5,6) [label=left:$p$]{};
	      		
	      		\node (q) at (-3.5,6) [label=left:$q$]{};
	      		
	      		\node (r) at (-2.5,6) [label=left:$r$]{};
	      		\node (s) at (-0.5,6) [label=right:$s$]{};
	      		
	      		\node (t) at (2.5,6) [label=left:$t$]{};
	
	      		\draw[very thick] (a) -- (b);
	      		\draw[red,ultra thick] (a) -- (c);
	      		\draw[green,ultra thick] (a) -- (d);
	      		\draw[very thick] (a) -- (e);
	
	     	 	     \draw[very thick] (b) -- (f);
	      		\draw[red,ultra thick] (b) -- (g);
	      		\draw[very thick] (b) -- (h);
	      		
	      		\draw[very thick] (c) -- (i);
	      		\draw[green,ultra thick] (c) -- (j);
	      		
	      		\draw[very thick] (d) -- (k);
	      		\draw[red,ultra thick] (d) -- (l);
	      		\draw[very thick] (d) -- (m);
	      		
	      		\draw[red,ultra thick] (e) -- (n);
	      		\draw[very thick] (e) -- (o);
	      		
	      		\draw[green,ultra thick] (g) -- (p);
	      		
	      		\draw[very thick] (i) -- (q);
	      		
	      		\draw[red,ultra thick] (j) -- (r);
	      		\draw[very thick] (j) -- (s);
	      		
	      		\draw[very thick] (l) -- (t);
      		\end{scope}
    	\end{tikzpicture}
	\end{center}
	\label{fig:red-green tree}
	\caption{A \red\ - \green-colored tree. With this tree at the input, Algorithm 1, when using the topological order that is represented by the alphabetical labeling of the vertices, returns the linear order $$c~i~q~r~j~s~g~p~n~a~b~e~f~h~o~l~t~d~k~m,$$ while Algorithm 2 returns the linear order $$a~b~f~h~d~k~m~l~t~e~o~n~c~i~q~j~s~r~g~p.$$}
\end{figure}

\begin{tool}[Color Detection]\label{prop:mega-tool}
Let $T$ be a rooted tree with some of its edges being colored either with \red or with \green. 
Then there is a branching program $\mathcal{B}$ depending on five linear orders on $V(T)$ such that 
for every queried $x,y\in V(T)$ with $x\wedge y$ being strictly below $x$ (resp.~$y$) in $T$, $\calB$ outputs $1$ if and only if 
the first colored edge on $[x\wedge y,x]$ (resp.~$[x\wedge y,y]$) is \red.
\end{tool}
\begin{proof}
The two setups are clearly symmetric, so in the proof, we concentrate only on the case when the queried vertices $x,y$ are such that $x\wedge y$ is strictly below $x$.

The root node $N_{\text{root}}$ of $\mathcal{B}$ establishes first the relative position of $x$ and $y$ in $T$.
This can be done with two linearorders on $V(T)$: the left-to-right depth first search ordering $\pi_L$ of $V(T)$ and the right-to-left depth-first-search ordering $\pi_R$ of $V(T)$. 
By the assumption $x\wedge y$ being strictly below $x$ in $T$, there are three possible outcomes:
\begin{itemize}
\item[--]  $y$ is strictly below $x$ in $T$ $\iff$ $y<x$ in both $\pi_L$ and $\pi_R$;
\item[--]  $x$ is left of $y$ in $T$ $\iff$ $x< y$ in $\pi_L$ and $y<x$ in $\pi_R$;
\item[--]  $y$ is left of $x$ in $T$ $\iff$ $y< x$ in $\pi_L$ and $x<y$ in $\pi_R$. 
\end{itemize}
$N_{\text{root}}$ has three children $N_{y\text{ below }x}$, $N_{x\text{ left of }y}$ and $N_{y\text{ left of }x}$, 
and it simply outputs the evaluation of the child responsible for the detected case of the relative position of $x$ and $y$ in $T$.
We will show that each child needs only one extra linear order to output the correct value.
We start with $N_{y\text{ below }x}$.

\begin{claim}
Suppose $y$ is strictly below $x$ and let $\pi$ be the linear order on $V(T)$ produced by Algorithm~\ref{Detection-Algo:y-below-x}.
Then $x<y$ in $\pi$ if and only if the first colored edge on the path $[y,x]$ in $T$ is \red.
\end{claim}

\begin{algorithm-hbox}[!ht]
	\caption{Color detection for $y$ being below $x$ in $T$}\label{Detection-Algo:y-below-x}
	\Proc{process($v$)}{
		$C(v)=\{u\in T_v\mid$ the path $[v,u]$ has no colored edges$\}$\; 
		$R(v)=\{u\in T_v\mid$ the path $[v,u]$ has a unique colored edge,\\
		\hskip32mm this edge is the last edge and it is \red$\!\!\}$\; 
		$G(v)=\{u\in T_v\mid$ the path $[v,u]$ has a unique colored edge,\\ 
		\hskip32mm this edge is the last edge and it is \green$\!\!\}$\;
		\vspace{5pt} 
		\For{$u\in R(v)$}{ process($u$) } \vspace{5pt}
		$L=$ list of all vertices in $C(v)$ in a topological order in $T$\;
		append $L$ to $\pi$\; \vspace{5pt} 
		\For{$u\in G(v)$}{ process($u$) } }
	$\pi=$\ empty\;
	\textrm{process}($\root T$)\;
	return $\pi$\;
\end{algorithm-hbox}

\begin{proof}[Proof of the Claim] 
We start with two simple observations.
\begin{itemize}
\item[--] During the execution of process($\root T$), every vertex of $T$ is a member of $C(v)$ for a unique $v\in V(T)$.
\item[--] A call of process($v$) appends all elements of $T_v$ to $\pi$ before returning.
\end{itemize}
Now consider $x,y$ such that $y$ is strictly below $x$ in $T$, and let $z_y$ be the node of $T$ for which $y\in C(z_y)$ during the algorithm. Also let $L_y$ be the list of all vertices in $C(z_y)$ in topological order. We distinguish three cases.
\begin{enumerate}
\item \emph{The path $[y,x]$ has no colored edges.} In this case $x\in C(z_y)$, and in $L_y$ we have $y$ before $x$. Therefore $y<x$ in $\pi$.
\item \emph{The first colored edge $(v,w)$ on the path $[y,x]$ is \green.} In this case $w\in G(z_y)$, hence the process($z_y$) appends the list $L_y$ containing $y$ before calling process($w$). However, as process($w$) is the one which puts $x$ in the linear order, we again have $y<x$ in $\pi$.
\item \emph{The first colored edge $(v,w)$ on the path $[y,x]$ is \red.} In this case $w\in R(z_y)$ and hence the process($z_y$) appends call process($w$), and puts $x$ in $\pi$, before appending the list $L_y$ containing $y$. Now we have $x<y$ in $\pi$.
\end{enumerate} 
\end{proof}

We move on to  $N_{x\text{ left of }y}$ .

\begin{claim}
 Suppose $x$ is left of $y$ and let $\pi$ be the linear order on $V(T)$ produced by Algorithm~\ref{Detection-Algo:x-left-of-y}.  Then $y<x$ in $\pi$ if and only if the first colored edge on the path $[x\wedge y,x]$ in $T$ is \red.
\end{claim}

\begin{algorithm-hbox}[!ht]
  \caption{Color detection for $x$ being left of $y$ in
    $T$}\label{Detection-Algo:x-left-of-y}
\Proc{process($v$)}{
append $v$ to $\pi$\;
\For{each $w$ child of $v$ in $T$ taken in left-to-right-order}
{
\If{$(v,w)$ is uncolored}
{process($w$)}
\If{$(v,w)$ is \red}
{S.push($w$)}
\If{$(v,w)$ is \green}
{S.push($v$)
\qquad\qquad\qquad\textit{// marking the beginning of the local stack of $v$}\;
process($w$)\;
\While{S.top()\ $\neq v$}
{process(S.pop())}
S.pop()\quad\qquad\qquad\qquad\textit{// taking off the marker}\;
}
}
}
$\pi=\emptyset$\;
S\ $= \textrm{empty stack}$\;
\textrm{process}($\root T$)\;
\While{S\emph{ is non-empty}}
{process(S.pop())}
return $\pi$\;
\end{algorithm-hbox}

\begin{proof}[Proof of the Claim]
We start the proof again with some simple observations.
\begin{itemize}
\item[--] For each vertex $v$ in $T$, there is a unique call of \textrm{process}$(v)$, and so the resulting $\pi$ is really a linear order on
$V(T)$.
\item[--] If $u,v,u',v'$ are vertices of $T$ such that $u\in T_{u'}$ and $v\in T_{v'}$ and $u'$ is on the current stack when \textrm{process}$(v')$ is called, then $v < u$ in $\pi$.
\item[--] If $[u,u']$ is a path of uncolored edges in $T$ and $v$ is a vertex that is above $u$ and right of $u'$, then $u'<v$ in $\pi$.
\item[--] If $(v,w)$ is a \green edge with $v$ below $w$, then the local stack of $v$ makes the procedure behave as if Algorithm~\ref{Detection-Algo:x-left-of-y} had been called for the tree $T_w$. In particular it appends all the vertices of $T_w$ in a consecutive block of $\pi$.
\end{itemize}
Now consider $x,y$ such that $x$ is left of $y$ in $T$. We again distinguish three
cases.
\begin{enumerate}
\item \emph{The path $[ x\wedge y,x]$ has no colored edges.} Then $x < y$ in $\pi$ follows from the third observation above.
\item \emph{The first colored edge $(v,w)$ on the path $[ x\wedge y,x]$ is \green.} From the third observation above, we obtain $v < y$ in $\pi$. By the fourth observation, the call of \textrm{process}$(w)$ appends all vertices of $T_w$, including $x$, to $\pi$ in a consecutive block before the processing of the local stack of $v$ is finished, and so before $y$ is touched. This implies that  $x < y$ in $\pi$.
\item \emph{The first colored edge $(v,w)$ on the path $[ x\wedge y,x]$ is \red.} Then $w$ is put on the stack when processing $v$ and remains on the stack until all children of $ x\wedge y$ have been processed. If $v'$ is the child of $ x\wedge y$ with $y\in T_{v'}$, then $w$ is on the stack when \textrm{process}$(v')$ is called. The second observation above shows that in this case $y < x$ in $\pi$.
\end{enumerate} 
\end{proof}	

As the cases '$x$ left of $y$' and '$y$ left of $x$' are clearly symmetric, this finishes the proof of the existence of Tool \ref{prop:mega-tool}.
\end{proof}


\section{The proof setup}\label{sec:families}

Let $P$ be a poset with cover graph $G$ and suppose that $\tw(G) \leq k$. Fix a tree-decomposition $(T,\{B_t\}_{t\in V(T)})$ of $G$ of width at most $k$. We imagine the tree $T$ to be rooted and being drawn upwards with lowest vertex the root.  In particular at every vertex $t$ there is a fixed left-to-right order of its children. For $z\in P$ let $\root{z}$ denote the root (i.e.\ the lowest vertex) of the subtree of $z$. Massaging a bit the tree-decomposition, we can assume that the vertices $\root{z}$, $z\in P$ are all distinct.

First, we apply a standard greedy coloring procedure to the elements of $P$: Fix any ordering of the elements of $P$ such that if $\root{z}$ is below $\root{z'}$ in $T$, then $z$ is before $z'$ in the ordering. Then, along this ordering, color an element $z\in P$ with the least possible color that does not appear in $B_{\root{z}}$. In this way, the elements from the same bag will have distinct colors, and we clearly use at most $k+1$ colors. Denote this coloring by $c:P\to[k+1]$. For a vertex $t\in V(T)$ and $i\in[k+1]$, if there is a (unique) element $z\in B_t$ of color $i$, then we call it the \emph{representative} of color $i$ at $t$ and denote it by $\rep i t$, otherwise we say that $\rep i t$ is undefined.

For a vertex $t\in V(T)$ and an element $z\in P$ such that $\root{z}$ is below $t$ in $T$, we define the vector $\ve(z,t)$ of length $k+1$ so that for $i\in [k+1]$, the $i^{\text{th}}$ coordinate is
\[
\ve_i(z,t)=
\begin{cases}
< & \text{ if } z<\rep i t \text{ in $P$},\\ 
> & \text{ if } z>\rep i t \text{ in $P$},\\
\parallel & \text{ if } z\parallel\rep i t \text{ in $P$},\\ 
= & \text{ if } z=\rep i t,\\ 
\ast &\text{ if } \rep i t\text{ is undefined}.
\end{cases}
\]
Now, for a vertex $t\in V(T)$ we define 
\[ 
\act t=\set{\ve(z,t) \mid z\in P \text{ such that }\root z \text{ is below $t$ in $T$}}.
\]
Recall that above we allow $\root z=t$. Note that in general there are at most $5^{k+1}$ such possible vectors, so in particular for every $t\in V(T)$, we have $|\act t|\leq 5^{k+1}$.

Next we define an auxiliary directed graph $D$ accompanying our fixed tree-decomposition (together with a fixed coloring of $P$), that will play a key role in
the remaining argument. The vertex set of $D$ is $\bigcup_{t\in V(T)} D_t$, where $D_t=\set{t}\times\act t$, and there is an edge from
$(t,\ve (z,t))$ to $(t',\ve (z,t'))$ for every $t,t'\in V(T)$ with $t$ being a parent of $t'$ in $T$ and every $z\in P$ with $\root z$ being below $t$ in $T$. 

\begin{lemma}\label{lem:dag}
For every edge $tt'\in E(T)$ with $t$ being the parent of $t'$ in $T$ and every $d\in D_t$, the vertex $d$ has exactly one out-neighbor in $D_{t'}$
\end{lemma}

\begin{proof}
Let $t,t'\in V(T)$ with $t$ being the parent of $t'$. To prove the lemma, we need to show that for every two distinct elements $x,y\in P$ with $\root x$ and $\root y$ being below $t$ in $T$, and $\ve (x,t)=\ve (y,t)$, we have
\[
\ve(x,t') = \ve(y,t').
\]
To do so take an arbitrary $i\in [k+1]$. We distinguish three cases.
\begin{enumerate}
\item \emph{$\rep i {t'}$ is undefined.} Then by default we have $\ve_i(x,t') = \ast =\ve_i(y,t')$.
\item \emph{$\rep i {t'}=\rep i {t}$ are both defined and they are equal.} Then $\ve_i(x,t') = \ve_i(x,t) =\ve_i(y,t) =\ve_i(y,t')$.
\item \emph{$\rep i {t'}$ is defined and $\rep i {t}$ is undefined or both of them are defined but are different.} Let $z'=\rep i {t'}$. Then $t'=\root{z'}$, in particular $z'\neq x,y$ and hence $\ve_i(x,t'),\ve_i(y,t')\in \{<,>,\parallel\}$. Suppose first that $\ve_i(x,t')=`<`$. This by definition means that $x<z'$ in $P$.  Let us consider a cover chain of this relation in $P$.  By the properties of a tree-decomposition this chain must contain an element $z$ with $x \leq z < z'$ in $P$ such that $z\in B_t\cap B_{t'}$. Then $z$ is a representative of some color $i_0\neq i$  both at $t$ and at $t'$, so by the previous case we know that $\ve_{i_0}(x,t') =\ve_{i_0}(y,t')$. This implies $y \leq z < z'$ in $P$ and so we conclude $\ve_i(y,t')= `<`$ as required.  Along the very same line, we can prove that if $\ve_i(x,t')=`>`$ then $\ve_i(y,t')=`>`$. However as $\ve_i(x,t')$ and $\ve_i(y,t')$ can only take three possible values, this already finishes the proof of this case.
\end{enumerate}

\end{proof}

We define a coloring $c_D: V(D)\to [5^{k+1}]$ as follows.
Let $t_0,t_1,\dots,t_m$ be the vertices of $T$ in a topological order.
Color all vertices from $D_{t_0}$ with distinct colors.
Now for $i\in [m]$ assuming that vertices in $\bigcup_{j<i} D_{t_j}$ are already colored 
we color $D_{t_i}$ as follows:
\begin{enumerate}
\item for every $d\in D_{t_i}$ such that $d$ has at least one incoming edge in $D$ we color $d$ with the least color used on its in-neighbors;
\item once all vertices in $D_{t_i}$ with at least one incoming edge are colored we color the remaining vertices in $D_{t_i}$ with distinct colors not used so far on $D_{t_i}$.
\end{enumerate}
\noindent Lemma~\ref{lem:dag} guarantees that for each $t\in V(T)$, the vertices in $D_t$ all have distinct $c_D$-colors.
Moreover, for every directed path $F$ in $D$ the $c_D$-colors of the vertices along the path are non-increasing.
The decreasing sequence of $c_D$-colors we get after removing the repetitions is called the \emph{signature} of $F$.

From now on, we will forget for a while about the underlying poset $P$, and we will only concentrate on exploring the dag $D$.

Let $d\in V(D)$ and let $c_D(d)=\gamma \in [5^{k+1}]$. Given $d$ and $\gamma$, by Lemma~\ref{lem:dag}, there is a unique directed subgraph of $D$ which is a tree rooted at $d$ and spans all the vertices of $D$ that can be reached on a path from $d$ in $D$ with all vertices of the path $c_D$-colored with $\gamma$. Denote this tree by $\tree(d,\gamma)$.

For a subgraph $D'$ of $D$, we define $\proj{D'}$ to be a subgraph of $T$ spanned by 
\[
\set{t\in V(T)\mid  D_t\cap D'\neq \emptyset}.
\]
Next, as a further preparation for our branching program, we construct inductively for every decreasing sequence $\Gamma$ of colors from $[5^{k+1}]$ and
every ternary sequence $\alpha$ over $\set{0,1,2}$ of length $|\alpha|=|\Gamma|-1$ a family $\calF_{\Gamma}^\alpha$ of subtrees of $T$. As a basis for this construction we have a family $\calF_{(\gamma)}^{\emptyset}$ for each $\gamma\in[5^{k+1}]$. This family is defined to be 
\[
\calF_{(\gamma)}^{\emptyset} =\calF_{\gamma} = \set{\proj{\tree(d,\gamma)}\mid \text{$d$ is a source in $D$ with $c_D(d)=\gamma$}}.
\] 
An essential property of the families will be that the trees in $\calF_{\Gamma}^\alpha$ are pairwise disjoint for every decreasing sequence $\Gamma$ over $[5^{k+1}]$ and every ternary sequence $\alpha$ over $\set{0,1,2}$ with $|\alpha|=|\Gamma|-1$.

\begin{claim}
The family $\calF_{(\gamma)}^{\emptyset}$ contains pairwise disjoint subtrees of $T$ for every $\gamma \in [5^{k+1}]$.
\end{claim}
\begin{proof}
Suppose to the contrary that $\calF_{\gamma}$ contains distinct subtrees $Q_1,Q_2$ of $T$ with root vertices $t_1,t_2$, respectively, that share a vertex $t$. In particular $t$ must be above both $t_1$ and $t_2$, implying that we either have $t_1$ below $t_2$  in $T$ or the other way round. Without loss of generality, assume that we have $t_1$ below $t_2$ in $T$. For $i=1,2$ let $d_i\in D_{t_i}$ be such that $Q_i=\proj{\tree(d_i,\gamma)}$. As $t$ is in both $Q_1$ and $Q_2$, there must exist a vertex $d\in D_t$ of $c_D$-color $\gamma$, and this vertex has to belong to $D_i:=\tree(d_i,\gamma)$ for $i=1,2$. In particular this implies that for $i=1,2$, there is a directed path from $d_i$ to $d$ in $D_i$ with all vertices along the path being of $c_D$-color $\gamma$.  As $t_1$ is below $t_2$ in $T$, the directed path from $d_1$ to $d$ has to go through $D_2$; however, the only vertex there of $c_D$-color $\gamma$ is $d_2$. This, unless $d_1=d_2$ and hence $Q_1=Q_2$, contradicts the fact that $d_2$ is a source vertex and so has no incoming edge.
\end{proof}

Now suppose that for a decreasing sequence $\Gamma$ over $[5^{k+1}]$ and a ternary sequence $\alpha$ over $\set{0,1,2}$ with $|\alpha|=|\Gamma|-1$, we are already given the family $\calF_{\Gamma}^\alpha$. Let $\gamma'$ be the last color in $\Gamma$, and let $\gamma\in [5^{k+1}]$ with $\gamma<\gamma'$. Then we define $\calF_{\Gamma \gamma}^{\alpha 0} := \calF_{\gamma}$. In particular, by the previous claim, we have that the subtrees in this family are pairwise disjoint. 

To construct the two other families $\calF_{\Gamma\gamma}^{\alpha 1}$ and $\calF_{\Gamma\gamma}^{\alpha 2}$ we will use an intermediate family $\calF_{\Gamma\gamma}^{\alpha+}$, which is produced by Algorithm~\ref{algo:families-construction-x-below-y}. For the description of this algorithm we need additional notation. For a color $\gamma\in [5^{k+1}]$, we call a vertex $t$ in $T$ a \emph{$\gamma$-break} if no vertex in $D_t$ has $c_D$-color $\gamma$. For an edge $tt'$ of $T$, $t$ being the parent of $t'$, and a color $\beta$ with $\beta<\gamma$, we say that $\gamma$ \emph{merges into} $\beta$ on $tt'$ if
\begin{enumerate}
	\item there is a vertex $d\in D_t$ of color $\gamma$ and a vertex $d'\in D_{t'}$ of color $\beta$, and
	\item there is an edge from $d$ to $d'$ in $D$.
\end{enumerate}

Finally, recall that for two vertices $u,v\in V(T)$ the unique path between them in $T$ is denoted by $[u,v]$.

\begin{algorithm-hbox}[!ht]
	\caption{Construction of $\calF_{\Gamma\gamma}^{\alpha+}$ given $\calF_{\Gamma}^{\alpha}$ 
	with $\Gamma\neq\emptyset$ and $\gamma'$ being the last color of $\Gamma$}
	\label{algo:families-construction-x-below-y}
	$\calF_{\Gamma\gamma}^{\alpha+}\gets
	\emptyset$\; 
	
	\For{each $Q$ in
		$\calF_{\Gamma}^{\alpha}$}{
		$Q^+\gets\emptyset$, \quad $r\gets \text{root of $Q$ in $T$}$
		
		\For{each $t$ in $Q$}{
		\For{each $t'$ child of $t$ such that\\
			\qquad\qquad(a) $t'$ is not in $Q$,\\
			\qquad\qquad(b) $[r,t]$ contains a $\gamma$-break,\\
			\qquad\qquad(c) $\gamma'$ merges into $\gamma$ at $tt'$}{ $Q^+ \gets
			Q^+ \cup [r,t'] \cup \proj{\tree(d',\gamma)}$,\\ 
			\quad where $d'$ is the vertex of color $\gamma$ in $D_{t'}$}
		}
		\If{$Q^+\neq\emptyset$}{ add $Q^+$ to
			$\calF_{\Gamma\gamma}^{\alpha+}$}
	} return
	$\calF_{\Gamma\gamma}^{\alpha+}$
\end{algorithm-hbox}

\begin{figure}
\centering
\includegraphics[scale=1.5]{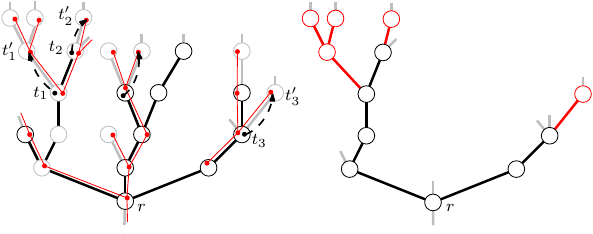}
\caption{\label{fig:tree-extension}Iteration step of Algorithm 3. On the left $Q$ is drawn with bolded lines, the color $\gamma$ is represented by red, and dashed arrows mark the places where $\gamma'$ merges into $\gamma$. On the right $Q^{+}$ is drawn with its primal section in black and its extended section it red.}
\end{figure}

The following claim shows that $\calF_{\Gamma\gamma}^{\alpha+}$ can be split into two families $\calF_{\Gamma\gamma}^{\alpha1}$ and $\calF_{\Gamma\gamma}^{\alpha2}$ such that each of them consists of pairwise disjoint trees.

\begin{claim}
The family $\calF_{\Gamma\gamma}^{\alpha+}$, produced by Algorithm~\ref{algo:families-construction-x-below-y}, can be split into two parts, each consisting of pairwise disjoint subtrees of $T$.
\end{claim}
\begin{proof} 
Every subtree $Q^+\in \calF_{\Gamma\gamma}^{\alpha+}$ produced by Algorithm~\ref{algo:families-construction-x-below-y} comes from some tree $Q\in\calF_{\Gamma}^{\alpha}$. Based on this, we split each such $Q^+$ into two sections, the \emph{primal section} $Q^+\cap Q$ and the \emph{extended section} $Q^+\setminus Q$.
	
We show that for every two distinct trees $Q_1^+,Q_2^+\in \calF_{\Gamma\gamma}^{\alpha+}$ both their primal and their extended sections are disjoint, respectively. This is immediate for the primal sections, as those are subgraphs of some disjoint members of $\calF_{\Gamma}^{\alpha}$.

Now consider the extended sections of $Q_1^+,Q_2^+$, and in order to get a contradiction, suppose that some vertex $t$ is in both of them. For $i=1,2$ let $Q_i$ be the pre-image of $Q_i^+$ in $\calF_{\Gamma}^{\alpha}$ with root vertex $r_i$.  Note that $r_i$ is also the root vertex of $Q_i^+$. Since for $i=1,2$, the vertex $t$ is in the extended section of $Q_i^+$, there must be a vertex $t_i$ of $Q_i$ and a child $t_i'$ of $t_i$ such that $t \in \proj{\tree(d_i',\gamma)}$, where $d_i'$ is the vertex of color $\gamma$ in $D_{t_i'}$; moreover there is also a $\gamma$-break at some vertex $t_i''$ on $[r_i,t_i]$. On the other hand, as members of $\calF_{\Gamma}^{\alpha}$, the subtrees $Q_1$ and $Q_2$ are disjoint; hence $[r_1,t_1]$ and $[r_2,t_2]$ need to be disjoint intervals on the path from the root of $T$ to $t$. Assuming without loss of generality that $t_1$ is below $r_2$ in $T$, we see that $t_2''$ is a $\gamma$-break in $[t_1,t]$, which contradicts the fact that  $t \in \proj{\tree(d_1',\gamma)}$.   
	
Hence, the intersection graph of the family $\calF_{\Gamma}^{\alpha+}$ is a chordal graph (as every intersection graph of subtrees of a tree is chordal) with clique number two and so it is two-colorable. Now a two-fcoloring induces a partition of $\calF_i^{\alpha+}$ into two families such that each consists of pairwise disjoint subtrees of $T$.
\end{proof}

The families $\calF_{\Gamma}^{\alpha}$ form the base of the key subprogram of our branching program designed to prove Theorem~\ref{thm:main}. 

\begin{lemma}\label{lem:mega-lemma}
For every decreasing color-sequence $\Gamma=(\gamma_1,\dots,\gamma_{\ell})$,
there is a branching program $\calB_{\Gamma}$, depending on at most $3^{\ell+1}-1$ linear orders on $V(T)$, such that
for every queried $x,y\in V(T)$ with $x\wedge y$ below $x$ (resp.~ $x\wedge y$ below $y$) in $T$, 
$\calB_{\Gamma}$ outputs $1$ if and only if
there is a path from $D_{x\wedge y}$ to $D_{x}$ (resp.~from $D_{x\wedge y}$ to $D_{ y}$) in $D$ with signature $\Gamma$.
\end{lemma}

\begin{proof}
The two setups are clearly symmetric, so in the proof we concentrate only on the case when the queried vertices $x,y$ are such that $x\wedge y$ is strictly below $x$ in $T$.

For $1\leq i\leq \ell$ let $\Gamma_i$ denote the prefix of $\Gamma$ of length $i$. 
The root node $N_{\text{root}}$ of $\mathcal{B}_{\Gamma}$ starts with checking whether there is a vertex $d\in D_{x\wedge y}$ with $c_D$-color $\gamma_1$, which happens exactly if $x\wedge y$ is a vertex of some subtree $Q\in \calF_{(\gamma_1)}^{\emptyset}$. 
This can be verified with the Color Detection Tool (Tool~\ref{prop:mega-tool}) using five linear orders. 
Indeed, color the edges of $T$ as follows. For every $Q\in \calF_{(\gamma_1)}^{\emptyset}$ and every vertex $t$ of  $Q$, color all edges going from $t$ to its children in $T$ with color \red and all other edges of $T$ by \green.
Then $x\wedge y$ is a vertex of some $Q\in \calF_{(\gamma)}^{\emptyset}$ exactly if the first colored edge on $[x\wedge y,x]$ in $T$ is \red. 
If the answer to this question is no, then $N_{\text{root}}$ immediately returns $0$, otherwise let $d$ be the vertex in $D_{x\wedge y}$ with $c_D$
-color $\gamma_1$. The branching program now proceeds step-by-step verifying whether the consecutive colors of the signature of the path from $d$ to $D_{x}$ (which by Lemma~\ref{lem:dag} is unique) agree with those in $\Gamma$.

For this we define for every $1\leq i\leq \ell$ and $\alpha\in \{0,1,2\}^{i-1}$ a subprogram $N_{i,\alpha}$. 
The branching program is going to visit a sequence of these nodes starting with $N_{1,\emptyset}$. We always make sure that when $N_{i,\alpha}$ is visited, then the following invariant holds:
\begin{itemize}
\item[--] $\Gamma_i$ is a prefix of the signature of the path from $d$ to $D_{x}$; moreover,
\item[--] if $\Gamma_i$ is a \emph{proper} prefix of the path from $d$ to $D_{x}$ then there is a tree $Q\in \calF_{\Gamma_i}^\alpha$ such that $x\wedge y$ and $t$ are both vertices of $Q$, where $t$ is the unique vertex of $T$ such that the maximal subpath of the path from $d$ to $D_{x}$ with signature $\Gamma_i$ ends in $D_t$;
\item[--] otherwise when the signature of the path from $d$ to $D_{x}$ is equal to $\Gamma_i$, there is a tree $Q\in \calF_{\Gamma_i}^\alpha$ such that $x\wedge y$ and $x$ are both in $Q$.
\end{itemize}
Note that $N_{\text{root}}$ already checked that $x\wedge y$ is contained in some tree $Q\in \calF_{(\gamma_1)}^{\emptyset}$ that contains $\proj{\tree (d,\gamma_1)}$ as a subtree, and hence the invariant holds for $i=1$ and $\alpha=\emptyset$.

Now we continue with the description of $N_{i,\alpha}$  for $1\leq i < \ell$ and $\alpha\in\set{0,1,2}^{i-1}$.
This subprogram consists of two steps.

\noindent \textbf{Step 1.} 
The first task of $N_{i,\alpha}$ is to establish whether 
the next color on the path from $d$ to $D_{x}$ is $\gamma_{i+1}$ or not. 
This can be done using five linear orders with Color Detection Tool (Tool~\ref{prop:mega-tool}).
Indeed, color the edges of $T$ as follows.
For every $Q \in \calF_{\Gamma_i}^\alpha$, every vertex $t$ of $Q$
and every child $t'$ of $t$ in $T$ not belonging to $Q$, color the edge $tt'$ \red if $\gamma_i$ merges into
$\gamma_{i+1}$ at $tt'$, otherwise color the edge $tt'$ \green. Now, by the invariant that holds, the first colored edge on the path $[x\wedge y,x]$ is \red exactly if the path from $d$ to $D_{x}$ continues after its initial segment of signature $\Gamma_i$ with color $\gamma_{i+1}$.
Otherwise the path from $d$ to $D_{x}$ either changes to some different color after $\Gamma_i$ or its signature is exactly $\Gamma_i$. In both cases 
$N_{\text{root}}$ outputs $0$.

Note that at this point if, the program is still running then we already know that $\Gamma_{i+1}$
is a prefix of the signature of the path from $d$ to $D_{x}$. In particular, $\Gamma_{i}$ is a proper prefix. By the invariant, this implies that there is a tree $Q\in\calF_{\Gamma_i}^\alpha$ such that $x\wedge y$ and $t$ are both vertices of $Q$, where $t$ is the unique vertex of $T$ such that the maximal subpath of the path from $d$ to $D_{x}$ with signature $\Gamma_i$ ends in $D_t$.

\noindent \textbf{Step 2.} The second task of $N_{i,\alpha}$ is to decide which subprogram to continue with, i.e., which is the relevant family out of  $\calF_{\Gamma_{i+1}}^{\alpha0}$, $\calF_{\Gamma_{i+1}}^{\alpha1}$ and $\calF_{\Gamma_{i+1}}^{\alpha2}$.

The family $\calF_{\Gamma_{i+1}}^{\alpha0}$ is relevant if there is no $\gamma_{i+1}$-break on the path from $x\wedge y$ to the vertex $t$ in $Q$.
This again can be checked using five linear orders with the Color Detection Tool (Tool~\ref{prop:mega-tool}).
Indeed color the edges of $T$ as follows.
First for every $Q \in \calF_{\Gamma_i}^\alpha$, every vertex $t$ of $Q$ and every child $t'$ of $t$ in $T$ not belonging to $Q$, color the edge $tt'$ \green.
Then for each vertex $t\in V(T)$ which is a $\gamma_{i+1}$-break color all
edges from $t$ going to its children with \red (possibly overriding \green). Note that as $x\wedge y$ is in $Q$ and  $x$ is not, we necessarily have a colored edge on the path $[x\wedge y,x]$, and the first colored edge on this path is \green exactly if there is no $\gamma_{i+1}$-break on $[x\wedge y,t]$. In this case, $N_{i+1,\alpha0}$ is called next. Note that the required invariant holds by the construction of $\calF_{\Gamma_{i+1}}^{\alpha0}$.

Otherwise, if the first colored edge on $[x\wedge y,x]$ is \red, then there is an $\gamma_{i+1}$-break on the path $[x\wedge y,t]$.
In this case we know that $\calF_{\Gamma_i}^{\alpha+}$
contains a tree $Q^+$ which contains the projection of the maximal subpath of
the path from $d$ to $D_{x}$ with signature $\Gamma_{i+1}$. Note that in particular, it also contains $x\wedge y$.
Now it remains to decide whether this tree $Q^+$ belongs to 
$\calF_{\Gamma_{i+1}}^{\alpha1}$ or to $\calF_{\Gamma_{i+1}}^{\alpha2}$. 
Any tree in  $\calF_{i}^{\alpha+}=\calF_{i+1}^{\alpha,1} \cup \calF_{i+1}^{\alpha,2}$
has been produced as the offspring of some tree in $\calF_{\Gamma_i}^{\alpha}$. Based on this, one can partition $\calF_{\Gamma_i}^{\alpha}$ into three 
subfamilies: those with an offspring in $\calF_{\Gamma_{i+1}}^{\alpha1}$,
those with an offspring in $\calF_{\Gamma_{i+1}}^{\alpha2}$ and the rest. Note that $Q^+$ is the offspring of the unique tree $Q$ that contains $x\wedge y$, so with a repeated application of the Color Detection Tool (Tool~\ref{prop:mega-tool}), we can identify, each time using five linear orders, the subfamily (out of the first two)
of $\calF_{\Gamma_i}^{\alpha}$ containing $Q$. 
Indeed, first take the subfamily of $\calF_{\Gamma_i}^{\alpha}$ containing the trees with an offspring in $\calF_{\Gamma_{i+1}}^{\alpha1}$. 
For every tree $Q$ in this family and every vertex $t$ of $Q$, color all edges going from $t$ to its children in $T$ with \red and color all other edges of $T$ with \green. Now the first colored edge on $[x\wedge y, x]$ is \red exactly if $x\wedge y$ is in one of the trees from the first subfamily. After possibly repeating this for the second subfamily, we will know in which of them is $Q$, and based on that the program continues either with the subprogram $N_{i+1,\alpha1}$ or with $N_{i+1,\alpha2}$. Whichever is chosen, by the construction of  $\calF_{\Gamma_{i+1}}^{\alpha1}$ and $\calF_{\Gamma_{i+1}}^{\alpha2}$ the required invariant holds.

This completes the description of the subprograms $N_{i,\alpha}$ for $1\leq i< \ell$ and $\alpha\in \{0,1,2\}^{i-1}$. 

It remains to describe the finishing subprograms $N_{\ell,\alpha}$ with $\alpha\in\set{0,1,2}^{\ell-1}$. When $N_{\ell,\alpha}$ is called then
by the invariant we know that $\Gamma$ is the prefix of the signature of the path from $d$ to $D_{x}$, and so 
the subprogram $N_{\ell,\alpha}$ needs to check whether the signature of this path is actually equal to $\Gamma$. Again by the invariant
this happens exactly if $x\wedge y$ and $x$ belong to the same tree in $\calF_{\Gamma}^{\alpha}$.
This can again be checked using the Color Detection Tool (Tool~\ref{prop:mega-tool}) using five linear orders. Color all edges of $T$ that do not belong to any tree $Q\in\calF_{\Gamma}^\alpha$ with color \red and keep all the other edges of $T$ uncolored. Then the first edge on the path from $x\wedge y$ to $x$ is \red exactly if $x\wedge y$ and $x$ belong to the same tree in $\calF_{\Gamma}^{\alpha}$. In case they do, $N_{\text{root}}$ returns $1$, otherwise it returns $0$.

By construction the branching program $\calB_{\Gamma}$ clearly returns $1$ if and only if there is a path from $D_{x\wedge y}$ to $D_x$ in $D$ with signature $\Gamma$. What is still missing is to count the number of linear orders $\calB_{\Gamma}$ is depending on. Note that every linear order that $\calB_{\Gamma}$ uses comes from an application of the Color Detection Tool (Tool~\ref{prop:mega-tool}). The nodes $N_{\text{root}}$ and $N_{\ell,\alpha}$, $\alpha\in \{0,1,2\}^{\ell-1}$ both involve one, while the nodes $N_{i,\alpha}$, $1\leq i<\ell$, $\alpha\in\{0,1,2\}^{i-1}$ all involve four possible applications of Tool~\ref{prop:mega-tool}. In each of these applications two linear orders out of the total five are always the same, so the total number of linear orders appearing is at most
\[
5+\sum_{\alpha\in \{0,1,2\}^{\ell-1}} 3+ \sum_{i=1}^{\ell-1} \sum_{\alpha\in \{0,1,2\}^{i-1}}4\cdot 3=5+3^{\ell} + 4\sum_{i=1}^{\ell-1}3^i=5+3^{\ell} + 4\frac{3^{\ell}-3}{3-1}=3^{\ell +1}-1.
\]
\end{proof}

\section{The branching program for Theorem~\ref{thm:main}}\label{sec:main-program}

Our aim here is to design a branching program to answer queries of the form "is $x\leq y$?" for elements $x,y\in P$. 
The branching program should depend on linear orders on the poset $P$ and so we note that in the remaining part of this paper even though we will mostly construct linear orders on the vertices of the tree $T$ we understand without saying that the corresponding linear order on $P$ considered for the branching program is the one induced by the linear order restricted to the vertices of the form $\root{z}$ for
some $z\in P$.

The root node $N_{\text{root}}$ of this program first of all quickly separates the case $x=y$.
This can be done with two initial linear orders on $P$ where one is the reversal of the other. If $x\leq y$ in both linear orders then we must have $x=y$ and $N_{\text{root}}$ outputs $1$.
Therefore, in what follows, we assume that the queried elements $x$ and $y$ are distinct. 

The key ingredient behind the operation of the branching program is summarized in the following lemma. For $x,y\in P$ let $m_{xy}=\root{x}\wedge\root{y}$.

\begin{lemma}\label{lem:two-seq}
We have $x\leq y$ in $P$ if and only if there exists
\begin{itemize}
\item[--] a decreasing color sequence $\Gamma=(\gamma_1,\dots,\gamma_{\ell})$ 
together with a path $F_x$ in $D$ such that it has signature $\Gamma$
and $\proj {F_x}=[m_{xy},\root{x}]$;
\item[--] a decreasing color sequence $\Delta=(\delta_1,\dots,\delta_{m})$ with $\delta_1=\gamma_1$ 
together with a path $F_y$ in $D$ such that it has signature $\Delta$
and $\proj {F_y}=[m_{xy},\root{y}]$;
\item[--] a vertex $d_x=(\root{x},v^x)\in D_{\root{x}}$ such that $c_D(d_x)=\gamma_{\ell}$ and $v^x_{c(x)}='>'\text{ or }'\!='$;
\item[--] a vertex $d_y=(\root{y},v^y)\in D_{\root{y}}$ such that $c_D(d_y)=\delta_{m}$ and $v^y_{c(y)}='<'\text{ or }'\!='$.
\end{itemize}
\end{lemma}
\begin{proof}
First suppose $x\leq y$ in $P$, and let $x=z_1<z_2<\cdots<z_s=y$ be a cover chain for this relation in $P$ (note that here $s=1$ is possible if $x=y$). By the basic properties of a tree decomposition, we know that the union of the subtrees of $z_1,\dots,z_s$ forms a subtree of $T$ containing the path from $\root{x}$ to $\root{y}$. In particular, there must be an index $1\leq i\leq s$ such that $z_i\in B_{m_{xy}}$. Then $\root{z_i}$ must be below $m_{xy}$ in $T$, and so the vector $\ve(z_i,t)$ is defined for every vertex $t$ above $m_{xy}$. Let $F_x$ and $F_y$ be the paths in $D$ from $(m_{xy},\ve(z_i,m_{xy}))$ to $d_x=(\root{x},\ve(z_i,\root{x}))$ and from $(m_{xy},\ve(z_i,m_{xy}))$ to $d_y=(\root{y},\ve(z_i,\root{y}))$ containing only vertices of the form $(t,\ve(z_i,t))$. Then clearly $\proj {F_x}=[m_{xy},\root{x}]$ and $\proj {F_y}=[m_{xy},\root{y}]$. Now let $\Gamma=(\gamma_1,\dots,\gamma_{\ell})$ and $\Delta=(\delta_1,\dots,\delta_{m})$ be the signatures of $F_x$ and $F_y$ respectively. Here $\gamma_1=\delta_1$ is just the $c_D$-color of  $(m_{xy},\ve(z_i,m_{xy}))$. To finish this direction, just note that from $x\leq z_i\leq y$, it follows that  $\ve_{c(x)}(z_i,\root{x})='>'\text{ or }'\!='$ and $\ve_{c(y)}(z_i,\root{y})='<'\text{ or }'\!='$.

Now for the backwards implication, first note that as $\gamma_1=\delta_1$ the initial vertex of the paths $F_x$ and $F_y$ has to be the same vertex in $D_{m_{xy}}$, let us denote it by $d=(m_{xy},v)$. Now let $z\in P$ be such that $\root{z}$ is below $m_{xy}$ in $T$ and $v=\ve(z,m_{xy})$. Note that by Lemma~\ref{lem:dag}, all vertices along the paths $F_x$ and $F_y$ are of the form $(t,\ve(z,t))$. In particular, we have $v^x=\ve(z,\root{x})$ and $v^y=\ve(z,\root{y})$. However, then  $v^x_{c(x)}='>'\text{ or }'\!='$ and $v^y_{c(y)}='<'\text{ or }'\!='$ just mean that $x\leq z \leq y$ in $P$, as required.
\end{proof}

With Lemma~\ref{lem:two-seq} in place, we can continue with the description of the branching program. $N_{\text{root}}$ will have a child $N_{\Gamma,\Delta}$ for every pair of decreasing color sequences $\Gamma,\Delta$ with the same starting color, which will return $1$ exactly if the requirements from Lemma~\ref{lem:two-seq} are satisfied with these color sequences. $N_{\text{root}}$ then will call its children one-by-one until one of them returns a $1$, in which case $N_{\text{root}}$ also returns a $1$. Otherwise, when all of its children return a $0$, $N_{\text{root}}$ also returns a $0$.

What is left now is to describe how $N_{\Gamma,\Delta}$ works for given decreasing color sequences $\Gamma=(\gamma_1,\dots,\gamma_{\ell})$ and $\Delta=(\delta_1,\dots,\delta_{m})$ with $\gamma_1=\delta_1$. It will have two children $N_{>,\Gamma}$ and $N_{<,\Delta}$, whose task will be to check the existence of the pairs $(F_x,d_x)$ and $(F_y,d_y)$ from Lemma~\ref{lem:two-seq}, respectively, and so $N_{\Gamma,\Delta}$ will return a $1$ exactly if both of its children succeed.

Because of symmetry, here we will only concentrate on the description of $N_{>,\Gamma}$, everything stated translates naturally to the setting of $N_{<,\Delta}$. 

To start with, $N_{>,\Gamma}$ handles the question about the existence of $d_x$. 
This can be done using the Set Membership Tool (Tool~\ref{cor:set-memb}) using three linear orders. Indeed, set 

$S_{\gamma_{\ell}}=\set{z\in P\ \mid\ \exists \ d=(\root{z},v)\in D_{\root{z}}\text{ with }c_D(d)=\gamma_{\ell}\text{ and }v_{c(z)}='>'\text{ or }'\!='}$.

Then $d_x$ exists exactly if $x\in S_{\gamma_{\ell}}$. 
If this is not the case then $N_{>,\Gamma}$ returns $0$; otherwise it continues with handling  the question about 
the existence of $F_x$.

For this it next checks whether $m_{xy}=\root{x}$ or $m_{xy}$ is strictly below $\root{x}$.  This can be easily done by looking at the left-first-search and right-first-search order of the vertices of $T$. 

\textbf{Case 1 ($m_{xy}=\root{x}$)}. Note that we already know that $D_{\root{x}}$ contains a vertex of color $\gamma_{\ell}$, so in this case a suitable $F_x$ exists if and only if $\ell=1$, so in this case $N_{x,\Gamma}$ returns a $1$ if and only if $\ell=1$.

\textbf{Case 2 ($m_{xy}$ is strictly below $\root{x}$)}. In this case $N_{>,\Gamma}$ simply calls the subprogram $\mathcal{B}_{\Gamma}$ guaranteed by Lemma~\ref{lem:mega-lemma} for the vertices $\root{x}$ and $\root{y}$ and returns the same value as $\mathcal{B}_{\Gamma}$.

This finishes the description of the branching program. It is clear by the construction that it really returns $1$ if and only if $x\leq y$ in $P$. 

One final thing we need to do in order to prove Theorem~\ref{thm:main} is to count how many linear orders does this branching program use. For this, first note that for fixed $\Gamma$ (resp.~fixed $\Delta$) the child node $N_{>,\Gamma}$ (resp.~$N_{<,\Delta}$) of the node $N_{\Gamma,\Delta}$ is the same for every $\Delta$  (resp.~every $\Gamma$) -- we think of them as being identical copies of the same node -- and hence depend on the same set of linear orders, namely three linear orders related to the Set Membership Tool (Tool~\ref{cor:set-memb}), the left-first-search and right-first-search order of the vertices of $T$ (the same for every node) and the $3^{|\Gamma|+1}-1$ (resp.~$3^{|\Delta|+1}-1$) linear orders related to the subprogram $\calB_{\Gamma}$ (resp.~$\calB_{\Delta}$). In addition to this the root node uses two more linear orders to identify the case $x=y$, which in total gives at most
\begin{align*}
&  2+\sum_{\Gamma,\Delta:\delta_1=\gamma_1}\Big(3+(3^{|\Gamma|+1}-1)+3+(3^{|\Delta|+1}-1)\Big)+2\\
&\leq 4+\sum_{\Gamma}(3+3^{|\Gamma|+1}-1)+\sum_{\Delta}(3+3^{|\Delta|+1}-1)=4+2\sum_{\ell=1}^{5^{k+1}}\binom{5^{k+1}}{\ell}(2+3^{\ell+1})
\\
&=4+4\sum_{\ell=1}^{5^{k+1}}\binom{5^{k+1}}{\ell}+6\sum_{\ell=1}^{5^{k+1}}3^{\ell}=4+4(2^{5^{k+1}}-1)+6((3+1)^{5^{k+1}}-1)\\
&=6\cdot 4^{5^{k+1}}+4\cdot 2^{5^{k+1}}-6.
\end{align*}
linear orders, and so results $\bdim{P} \leq 6\cdot 4^{5^{k+1}}+4\cdot 2^{5^{k+1}}-6$. As the upper bound on $\bdim{P}$ only depends on $k$, this finishes the proof of Theorem~\ref{thm:main}.

\section{A connection to reachability labeling schemes}\label{sec:reachability}

Boolean realizers have a natural connection to labeling schemes for
reachability queries for families of directed graphs. 
A similar application was already observed by Gambosi, Ne{\v{s}}et{\v{r}}il and Talamo~\cite{GNT87}.
For a family $\vec{\calG}$ of digraphs, a \emph{labeling}
is a non-negative integer function $L$ that assigns for every digraph $\vec{G}\in \vec{\calG}$ a \emph{label}
$L(v,\vec{G})$ to each vertex $v$ of $\vec{G}$. A \emph{reachability decoder} is a function $f$ that given
two labels $\lambda_1,\lambda_2$, returns a binary value
$f(\lambda_1,\lambda_2)$.  Now, a pair $(L,f)$ is called a \emph{reachability
  labeling scheme}, if for every digraph $\vec{G}\in \vec{\calG}$ and every pair of vertices $u,v$ in $\vec{G}$ there is a path from $u$ to $v$ in $\vec{G}$ if and only if $f(L(u,\vec{G}),L(v,\vec{G}))=1$. 

Formally we see labels as binary
strings and given a label $L(v,\vec{G})$, we put $|L(v,\vec{G})|$ for its length as a binary string. 
Then the \emph{size} of a reachability labeling scheme $(L,f)$ (as a function of $n$) is 
\[
\max_{\vec{G}\in \vec{\calG}, |V(\vec{G})|=n}\max_{v\in V(\vec{G})}|L(v,\vec{G})|.
\] 
The central question in this area is to determine how small can reachability labeling schemes be for different families of graphs. A particularly interesting case is when $\vec{\calG}$ is the family of planar digraphs. Along this line, Thorup~\cite{Th04} presented a reachability labeling scheme of size $O(\log^2{n})$ for planar digraphs on $n$ vertices; however, it still remains a
challenge to answer the following question.

\begin{problem}
Is there a reachability labeling scheme of size $O(\log{n})$ for planar digraphs on $n$ vertices?
\end{problem}

We remark that a positive answer to a stronger version of Problem~\ref{problem:planar} would result, as argued in the observation below, in a positive answer for the above problem as well. 

\vspace{0.1cm}

\textbf{Observation.} Let $\calG$ be a family of (undirected) graphs which is closed under taking minors. Suppose that the boolean dimension of posets whose cover graph belongs to $\calG$ is bounded from above by some constant $k$ in a strong sense: the boolean formula participating in the boolean realizer can be chosen to be the formula $\phi$ for every poset. Note that in this case, the boolean realizers should contain exactly $k$ linear orders for every poset. Then there exists a reachability labeling scheme of size $k\cdot \log n$ for the family $\vec{\calG}$ of all digraphs whose undirected variant belongs to $\calG$.

Indeed, to define the appropriate labeling, let $\vec{G}$ be and arbitrary digraph on $n$ vertices from $\vec{\calG}$ . 
Note that for any reachability labeling, we may assume that vertices along a directed cycle all get the same label, so it is enough to deal with the acyclic 
digraph $\vec{G'}$ on $n'\leq n$ vertices, also belonging to $\vec{\calG}$, that we get from $\vec{G}$ by contracting all directed cycles.
Now that $\vec{G}'$ is acyclic, it gives rise to a poset $P$, whose elements are the vertices of $\vec{G'}$ and $x\leq y$ in $P$ if and only if there is a directed path from $x$ to $y$ in $\vec{G'}$. 
So deciding whether $x\leq y$ in $P$ is the same as deciding whether $y$ is reachable from $x$ in $\vec{G'}$. 
As the cover graph of $P$ is the undirected version of $\vec{G'}\in \vec{\calG}$, and as such belongs to $\calG$, by assumption we have $\bdim P\leq k$. 
To finally define the labeling, take the $k$ linear orders that together with $\phi$ form a boolean realizer for $P$, and for an element $x\in P=V(\vec{G'})$ put into $L(x,\vec{G'})$ the respective positions of $x$ in the $k$ linear orders. 
Then the size of any label is $k\cdot \log n'\leq k\cdot \log n$, and given two elements  $x,y\in P=V(\vec{G'})$ from their labels, we can extract their relative position in the $k$ linear orders, and, using $\phi$, decide whether $x\leq y$ in $P$, i.e. whether there is a directed path from $x$ to $y$ in $\vec{G'}$.  

\vspace{0.1cm}

Note that within the proof of Theorem~\ref{thm:main}, the formula (or the branching program) witnessing small boolean dimension does not depend on the poset $P$.
Therefore, by the Observation above, we obtain a reachability labeling scheme of size $O(\log{n})$ for digraphs on $n$ vertices with bounded tree-width.
As we pointed out this recently, such a reachability labeling scheme can be obtained more directly using the elimination tree of $G$ witnessing the inequality $\td(G)\leq \tw(G)\log n$, where $\td(G)$ is the tree-depth of $G$.

\section{Posets with unbounded boolean dimension whose cover graphs have maximum degree $3$}\label{sec:topminor}

A poset $P$ is said to be an \emph{interval order} if there is an assignment $P\ni x\to[\ell_x,r_x]\subset\mathbb{R}$ such that 
for every $x,y \in P$ we have $x < y$ in $P$ if and only if $r_x < \ell_y$.
This assignment is often called an \emph{interval representation} of $P$.
A \emph{universal interval order} of order $n$ ($n\geq2$) is the poset $U_n$ with elements 
$\set{[i,j] \mid i,j\in\mathbb{N} \text{ and } 1\leq i < j \leq n}$ and 
$[i,j] < [k,\ell]$ in $U_n$ if $j < k$. The proof of the following claim is an easy modification of an argument from \cite{FHRT92}.

\begin{claim}
$\bdim{U_n} \geq \log\log\log n$.
\end{claim}
\begin{proof} 
Consider a boolean realizer of $U_n$ consisting of the linear orders $L_1,\ldots,L_d$ 
and 
a $d$-ary boolean formula $\phi$.
We will show that $d\geq\log\log\log n$.

The \emph{double-shift graph} $G_n$ of order $n$ is the graph with vertex set $\set{(i,j,k)\mid i,j\in\mathbb{N} \text{ and } 1\leq i< j<k\leq n}$ 
and $(i,j,k)$ being adjacent to $(j,k,\ell)$ for all integers $i,j,k,\ell$ with $1\leq i < j<k<\ell\leq n$.
It is well known that $\chi(G_n)\geq\log\log n$, see e.g.~\cite{FelPhD}.

For a triple of integers $(i,j,k)$ with $1\leq i < j < k \leq n$, we define $\psi(i,j,k)$ to be a binary sequence 
$(z_1,\ldots,z_d)$ of length $d$ so that $z_{\alpha}=1$ if $[i,j]<[j,k]$ in $L_{\alpha}$, and $z_{\alpha}=0$ if $[j,k]<[i,j]$ in $L_{\alpha}$, 
for every $\alpha\in[d]$.
We claim that $\psi$ is a proper coloring of $G_n$.
To see that, take four integers $i,j,k,\ell$ with $1\leq i<j<k<\ell\leq n$ and suppose that $\psi(i,j,k)=(z_1,\ldots,z_d)=\psi(j,k,\ell)$.
This means that for every $\alpha\in[d]$, if $z_{\alpha}=1$ then $[i,j]<[j,k]<[k,\ell]$ in $L_{\alpha}$, and if $z_{\alpha}=0$ then $[k,\ell]<[j,k]<[i,j]$ in $L_{\alpha}$.
Since $[i,j]$ and $[j,k]$ are incomaprable in $U_n$ we have that $\phi(z_1,\ldots,z_d)=0$.
On the other hand, $[i,j] < [k,\ell]$ in $U_n$ and therefore we must have $\phi(z_1,\ldots,z_d)=1$, a contradiction.
Thus $\psi$ is a proper coloring of $G_n$. 
Since $\psi$ takes at most $2^d$ values we get that $2^d\geq \log\log n$, as desired.
\end{proof}

It is easy to see that every interval order has a \emph{distinguishing} interval representation, 
which means an interval representation with all the endpoints $\ell_x$, $r_x$ for $x\in P$ being distinct.

Now consider an interval order $P$ together with a distinguishing representation. 
Let $e_1,\ldots,e_m$ be the values of all endpoints in the representation in increasing order.
Extend the poset $P$ to $P^+$ by introducing new elements represented by the intervals $[e_i + \frac{1}{5}(e_{i+1}-e_i), e_i + \frac{2}{5}(e_{i+1}-e_i)]$, $[e_i + \frac{3}{5}(e_{i+1}-e_i), e_{i} + \frac{4}{5}(e_{i+1}-e_i)]$ for every $i\in[m-1]$.
Note that every element of $P^+$ has at most three neighbors in the cover graph.
Therefore, the maximum degree of the cover graph of $P^+$ is at most $3$ and 
since $P$ is a subposet of $P^+$ we have $\bdim{P^+}\geq \bdim{P}$.

This way we can also extend the universal interval order $U_n$ to $U^+_n$, for every $n\geq2$, 
and obtain a family of posets with cover graphs of maximum degree $3$ and unbounded boolean dimension.

\bibliographystyle{plain}
\bibliography{posets-dimension}

\begin{thebibliography}{10}

\bibitem{BPSTT}
Fidel Barrera-Cruz, Thomas Prag, Heather Smith, Libby Taylor, and William~T.
  Trotter.
\newblock Comparing {D}ushnik-{M}iller dimension, boolean dimension and local
  dimension.
\newblock {\em Order}, 2019.
\newblock
  \href{https://doi.org/10.1007/s11083-019-09502-6}{https://doi.org/10.1007/s11083-019-09502-6}.

\bibitem{BGT}
Bartłomiej Bosek, Jarosław Grytczuk, and William~T. Trotter.
\newblock Local dimension is unbounded for planar posets.
\newblock submitted, \href{https://arxiv.org/abs/1712.06099}{arXiv:1712.06099}.

\bibitem{BF96}
Graham~R. Brightwell and Paolo~Giulio Franciosa.
\newblock On the {B}oolean dimension of spherical orders.
\newblock {\em Order}, 13(3):233--243, 1996.

\bibitem{DM41}
Ben Dushnik and Edwin~W. Miller.
\newblock Partially ordered sets.
\newblock {\em Amer. J. Math.}, 63:600--610, 1941.

\bibitem{FelPhD}
Stefan Felsner.
\newblock {\em Interval {O}rders: {C}ombinatorial {S}tructure and
  {A}lgorithms}.
\newblock PhD thesis, Technische Universit{\"a}t Berlin, 1992.

\bibitem{FHRT92}
Z.~F{\"u}redi, P.~Hajnal, V.~R{\"o}dl, and W.~T. Trotter.
\newblock Interval orders and shift graphs.
\newblock In {\em Sets, graphs and numbers ({B}udapest, 1991)}, volume~60 of
  {\em Colloq. Math. Soc. J\'anos Bolyai}, pages 297--313. North-Holland,
  Amsterdam, 1992.

\bibitem{GNT87}
G.~Gambosi, J.~Ne{\v{s}}et{\v{r}}il, and M.~Talamo.
\newblock Posets, boolean representations and quick path searching.
\newblock In Thomas Ottmann, editor, {\em Automata, Languages and Programming},
  pages 404--424, Berlin, Heidelberg, 1987. Springer Berlin Heidelberg.

\bibitem{GNT90}
Giorgio Gambosi, Jaroslav Nešetřil, and Maurizio Talamo.
\newblock On locally presented posets.
\newblock {\em Theoretical Computer Science}, 70(2):251 -- 260, 1990.

\bibitem{Kel81}
David Kelly.
\newblock On the dimension of partially ordered sets.
\newblock {\em Discrete Math.}, 35:135--156, 1981.

\bibitem{MMT}
Tam\'{a}s M\'{e}sz\'{a}ros, Piotr Micek, and William~T. Trotter.
\newblock Boolean dimension, components and blocks.
\newblock {\em Order}, 2019.
\newblock
  \href{https://doi.org/10.1007/s11083-019-09505-3}{https://doi.org/10.1007/s11083-019-09505-3}.

\bibitem{MW-pw}
Piotr Micek and Walczak Bartosz.
\newblock personal communication.

\bibitem{NP89}
J.~Ne\v{s}et\v{r}il and P.~Pudl\'ak.
\newblock A note on {B}oolean dimension of posets.
\newblock In {\em Irregularities of partitions ({F}ert\H od, 1986)}, volume~8
  of {\em Algorithms Combin. Study Res. Texts}, pages 137--140. Springer,
  Berlin, 1989.

\bibitem{Sew}
Michał Seweryn.
\newblock Improved bound for the dimension of posets of treewidth two.
\newblock {\em Discrete Mathematics}, 2019.
\newblock
  \href{https://doi.org/10.1016/j.disc.2019.111605}{https://doi.org/10.1016/j.disc.2019.111605}.

\bibitem{Th04}
Mikkel Thorup.
\newblock Compact oracles for reachability and approximate distances in planar
  digraphs.
\newblock {\em J. ACM}, 51(6):993--1024, 2004.

\bibitem{TM77}
William~T. Trotter, Jr. and John~I. Moore, Jr.
\newblock The dimension of planar posets.
\newblock {\em J. Combinatorial Theory Ser. B}, 22(1):54--67, 1977.

\bibitem{Uec16}
Torsten Ueckerdt.
\newblock proposed at Order \& Geometry workshop in Gułtowy, Poland, 2016.

\end{thebibliography}

\end{document}